\documentclass[12pt,oneside]{article}
\usepackage[T1]{fontenc}
\usepackage{a4}
\usepackage{amssymb}
\usepackage{mathrsfs}
\usepackage{amsmath}
\usepackage{amsfonts}
\usepackage{txfonts}
\usepackage[dvips]{graphicx}
\usepackage{color}
\usepackage{epsfig}
\usepackage{pstricks,pst-node}
\newtheorem{theorem}{Theorem}
\newtheorem{lemma}[theorem]{Lemma}
\newtheorem{proposition}[theorem]{Proposition}
\newtheorem{corollary}[theorem]{Corollary}

\newtheorem{definition}{Definition}
\newtheorem{remark}{Remark}%esto no estaba

\newtheorem{example}{Example}
\newcommand{\bdf}{\begin{df} \begin{rm}}
\newcommand{\edf}{\end{rm} \end{df}}

\newenvironment{proof}{{\bf Proof.}}{\hspace*{\fill} \rule{2mm}{2mm} \par \hspace{0.1mm}}

\title{Relations between edge removing and edge subdivision concerning domination number of a graph}

\author{Magdalena Lema\'nska  $^{1}$, Joaqu\'{\i}n Tey $^{2}$,  Rita Zuazua $^{3}$
\\
$^1${\small  Gdansk University of Technology, Poland,  magda\@@mif.pg.gda.pl ,}
\\
$^2${\small Universidad Aut\'onoma Metropolitana, Unidad Iztapalapa, M\'exico, jtey\@@xanum.uam.mx}
\\
$^3${\small Universidad Nacional Aut\'onoma de M\'exico, M\'exico, ritazuazua\@@ciencias.unam.mx}}

\date{}
\begin{document}

\maketitle

\begin{abstract}

Let $e$ be an edge of a connected simple graph $G$. The graph obtained by removing (subdividing) an edge $e$ from $G$ is denoted by $G-e$ $\left( G_{e}\right)$. As usual, $\gamma (G)$ denotes the domination number of $G$.
We call $G$ an SR-graph if $\gamma(G-e)=\gamma(G_{e})$ for any edge $e$ of $G$, and $G$ is an ASR-graph if
$\gamma(G-e)\neq \gamma(G_{e})$ for any edge $e$ of $G$.
In this work we give several examples of SR and ASR-graphs. Also, we characterize SR-trees and show that ASR-graphs are $\gamma$-insensitive.

{\bf \AmS \; Subject Classification:} 05C05, 05C69.
\end{abstract}

\section{Introduction and basic definitions}

Let $G$ be a connected simple graph. We denote by $V(G)$ and $E(G)$ the vertex set and the edge set of $G$, respectively. For a set $X\subseteq V(G),$ $G[X]$ is the subgraph induced by $X$ in $G$. The \emph{neighborhood} $N_{G}(u)$ of a vertex $u$ in $G$ is the set of all vertices adjacent to $u$, its \emph{closed neighborhood} is $N_{G}[u]=N(u)\cup \{u\}$  and \emph{the closed neighborhood of $X\subseteq V(G)$} is $N_{G}[X]=\bigcup\limits_{u\in X}N_{G}[u]$. A vertex $u$ in $G$ is a \emph{universal vertex} if $N_{G}[u]=V(G)$.

The \emph{external private neighborhood} of a vertex $u\in D$ with respect to $D \subseteq V (G)$ is the set $EPN_G(u,D) = N_G(u) - N_G[D - \{u\}]$.

The \emph{degree} of a vertex $u$ is $d_{G}(u)=|N_{G}(u)|$. A vertex $v$ in $G$ is a {\it leaf} if $d_G(v)=1$. A vertex $u$ is called a {\it support vertex} if it is adjacent to a leaf. We denote by
$\mathcal{L}_{G}\left(u\right) $ the set of leaves in $G$ adjacent to $u$. A support vertex $u$ is called a \emph {strong support vertex} if $|\mathcal{L}_{G}\left(u\right)|>1$. In the other case is called a \emph{weak support vertex}. The set of all support vertices of $G$ is denoted by $Supp(G)$.

For $X, Y \subseteq V(G)$, we say \textit{ $X$ dominates $Y$} (abbreviated by $X\succ Y$) if $Y \subseteq N_{G}[X]$.  If $N_{G}[X] = V (G)$, then $X$ is called a \emph{dominating set} of $G$ and we write $X\succ G$. If $Y = \{y\}$, we put $X\succ y$.

The \textit{domination number} of $G$, $\gamma (G)$, is the minimum cardinality among all dominating sets of $G$.  A minimum dominating set of a graph $G$ is called a \textit{$\gamma$-set} of $G$. We denote by $\Gamma (G)$ the set of all $\gamma$-sets of $G$. Let $e\in E\left( G\right) $ and $D\in \Gamma \left( G\right) $. If $\left\vert e\cap D\right\vert =1$, then $\overline{e\cap D}$  denotes the
vertex of $e$ not contained in $D$.

The undefined terms in this work may be found in \cite{b1,b2}.

\vspace{.3cm}

For an edge $e=uv$ of $G$, we consider the following two modifications of $G$.
\begin{itemize}
	\item Removing the edge $e$: we  delete $e$ from $G$ and obtain a new graph, which is denoted by $G-e$.
	\item Subdividing the edge $e$: we delete $e$, add a new vertex $w$ and add two new edges $uw$ and $wv$. The new graph is denoted by $G_{e}$.
\end{itemize}

$G$ is a \emph{$\gamma$-insensitive} graph if  $\gamma\left( G-e\right) =\gamma \left( G\right)$ for any edge $e$ of $G$. An edge $e$ of $G$ is called a $bondage$ $edge$ if $\gamma \left( G-e\right) >\gamma \left( G\right) $. We will use frequently the following characterization of a bondage edge of a graph given by Teschner in \cite{b3}.

\begin{theorem}\cite{b3}\label{teschner} An edge $e$ of a graph $G$ is a
bondage edge if and only if
$$
|e\cap D|=1 \text{ and } \overline{e\cap D}\in EPN(e\cap D,D) \text{ for any } D \text{ in } \Gamma (G).
$$
\end{theorem}

If an edge satisfies the above condition, then we say that it satisfies Teschner's Condition.

The relation between $\gamma (G)$ and $\gamma (G-e)$ was studied in several works. For example, in \cite{suffel,Walikar} the authors characterized graphs $G$ such that for every edge $e$ of $G,$ $\gamma(G-e)>\gamma(G)$. The $\gamma$-insensitive graphs were considered in \cite{Dutton,Haynes}.

On the other hand, influence of the subdivision of an edge on the domination number was studied for instance in \cite{haynes1,haynes2}.

In this paper we begin the study of the relation between the domination number of the graphs $G-e$ and $G_e$ for an edge $e$ of $G$. We start with the following remark and examples.

\begin{remark}
\label{obs1}
For any edge $e$ of a graph $G$ we have $$\gamma(G)\leq \gamma(G-e)\leq \gamma(G)+1\hbox{ and }  \gamma(G)\leq \gamma(G_{e})\leq \gamma(G)+1.$$
\end{remark}

As usual, $P_n$ and $K_n$ denote the path and the complete graph  of order $n$, respectively. Let $P_n=(v_1,v_2, \ldots, v_n)$.

\begin{itemize}
\item If $G=P_{6}$, then $2=\gamma(G-e)< \gamma(G_e)=3$ for $e=v_{3}v_{4}$ and  $\gamma(G-e)= \gamma(G_e)= 3$ for $e=v_1v_2$.
\item If  $G=P_{8},$ then $4=\gamma(G-e) > \gamma(G_e)=3$ for  $e=v_{4}v_{5}$ and $\gamma(G-e)=\gamma(G_e)= 3$ for $e=v_3v_4$.
\item If  $G=P_{7},$  then for any edge $e$, $\gamma(G-e)=\gamma(G_{e})=3$.
\item If $G=K_3$, then for any edge $e$, $1=\gamma(G-e) <  \gamma(G_{e})=2$.
\end{itemize}

The above situation motivates the following definition.

\begin{definition} Let $G$ be a  graph of order at least two.
\begin{enumerate}
\item We call $G$ a \emph{sub-removable} graph (shortly, SR-graph) if $\gamma(G-e)=\gamma(G_{e})$ for any edge $e$ of $G$.
\item We call $G$ an \emph{anti-sub-removable} graph (shortly, ASR-graph) if  $\gamma(G-e)\neq \gamma(G_{e})$ for any edge $e$ of $G$.
\end{enumerate}
\end{definition}

\begin{example}
The complete bipartite graph $G=K_{m,n}$, where $\max\{m,n\}>1,$ is an SR-graph. To show this we may suppose, without loss of generality, $m > 1$. Let $e=uv$ be an edge of $G$. If $n=1$, then $G$ is a star and $\gamma(G-e)=\gamma(G_e)=2$. Otherwise  $m,n>1$ and $\gamma(G)=2$. Moreover, $\{u,v\}\succ G-e$, $\{u,v\}\succ G_e$ and by Remark \ref{obs1}, $\gamma(G-e)=\gamma(G_e)=2$.
\end{example}

\begin{example}
The complete graph $G=K_{n}$, where $n\geq 3$, is an ASR-graph, because $\gamma(G-e)=1$ and $\gamma(G_e)=2$ for any edge $e$ of $G$.
\end{example}

Our paper is organized as follows: in Section 2 we give several examples of SR-graphs and show that every graph is an induced subgraph of an SR-graph. In Section 3 we characterize SR-trees and bondage edges in SR-trees. Finally, in Section 4 we characterize ASR-graphs with domination number one, give some properties of ASR-graphs, show that ASR-graphs are $\gamma$-insensitive and give an infinity family of ASR-graphs with arbitrary domination number.
\\

\section{Sub-removable graphs }

In this section we give some infinity families of sub-removable graphs and we show that every graph is an induced subgraph of an $SR$-graph.

\vspace{.3cm}

In the case of the path $P_n$ or the cycle $C_n$ of order $n$ its domination number is well known.

\begin{remark}\cite{b2}\label{dom-path} For $n\geq 1$,  $\gamma(P_n)=\gamma(C_{n})=\lceil \frac{n}{3}\rceil$.
\end{remark}

\begin{proposition}\label{SR-path}
The path $P_{n}$ is a sub-removable graph if and only if $n=3$ or $n\equiv 1 \pmod 3$ for $n\geq 4$.
\end{proposition}

\begin{proof} It is clear that the path $P_{2}=(v_1,v_2)$ is an anti-sub-removable graph. If $n=3,$ for any edge $e$ of $G=P_{3}$, $\gamma(G_{e})=\gamma(P_4)=2=\gamma(G-e)$.
Thus $P_3$ is an SR -graph.

\vspace{.3cm}

For $n\geq 4$, let $G=P_n=(v_1, ... ,v_n)$. We consider the next three cases.
\vspace{.3cm}

\textbf{Case 1.} If $n\equiv 0 \pmod 3,$ by Remark \ref{dom-path}, $\gamma(G_{e})=\gamma (P_{n+1})=\gamma(P_n)+1$ for any edge $e$. But, for $e=v_3v_4$, $\gamma(G-e)=\gamma(G),$ so $G$ is not an SR-graph.

\textbf{Case 2.} If $n\equiv 1 \pmod 3$, for any edge $e$,  $G-e= P_s \cup P_t$ where $s+t=n$. As $\gamma(G-e)=\gamma(P_s) + \gamma (P_t)$, by Remark \ref{dom-path}, $\gamma (G-e) = \lceil \frac{s}{3} \rceil +\lceil \frac{t}{3} \rceil = \lceil \frac{n}{3} \rceil= \lceil \frac{n+1}{3} \rceil = \gamma(G_e)$ and $G$ is an SR-graph.

\textbf{Case 3.} If $n\equiv 2 \pmod 3,$ by Remark \ref{dom-path}, $\gamma(G_e)=\gamma (P_{n+1})=\gamma(G)$ for any edge $e$. But, for $e=v_1v_2,$ $\gamma (G-e)=\gamma (P_{n-1}) + 1= \gamma (G)+1$, so $G$ is not an  SR-graph.

\end{proof}

\begin{proposition} Let $n \geq 3$. If $n\equiv 1,2 \pmod 3$, then the cycle $C_n$ is an SR-graph. Otherwise, is an ASR-graph.
\end{proposition}

\begin{proof}
If $G=C_n$, then for any edge $e\in E(G)$, $G-e=P_n$ and $G_{e}=C_{n+1}$.  If $n\equiv 1,2 \pmod 3$, by Remark \ref{dom-path},
$\gamma (G-e)=\gamma (P_n)= \lceil \frac{n}{3} \rceil = \lceil \frac{n+1}{3} \rceil= \gamma (C_{n+1})=\gamma (G_e)$.

In the other case, $n\equiv 0 \pmod 3$ and $\gamma (G-e)=\gamma (P_n)= \lceil \frac{n}{3} \rceil < \lceil \frac{n+1}{3} \rceil= \gamma (C_{n+1})=\gamma (G_e)$.
\end{proof}

Recall that we denote by $\Gamma (G)$ the set of all $\gamma$-sets of a graph $G$ and for a support vertex $u$,
$\mathcal{L}_{G}\left(u\right) $ is the set of leaves adjacent to $u$ in $G$.

\begin{remark}\label{Soporte fuerte}
If $u$ is a strong support vertex of a graph $G$, then for
any $D$  in $\Gamma (G)$, $u\in D$ and $\mathcal{L}_{G}\left( u\right) \cap D=\emptyset $.
\end{remark}

\begin{lemma}\label{SR fuerte}
Let $G$ be a graph and $e=uv$ be an edge of $G$ where
$v\in \mathcal{L}_{G}\left( u\right) $. If $u$ is a strong support vertex of
$G$, then $\gamma (G-e)=\gamma (G)+1=\gamma (G_{e})$.
\end{lemma}

\begin{proof} Let $D$ in $\Gamma (G)$. By Remark \ref{Soporte fuerte}, $e\cap D=u$, and $v\in EPN\left( u,D\right) $.
Therefore $e$ satisfies
Teschner's Condition and by Theorem \ref{teschner}, $e$ is a bondage edge, i.e., $\gamma
(G-e)=\gamma (G)+1$.

Let $D$ be a $\gamma$-set of $G_e$, assume $|D|=\gamma (G)$. Let $S=\{ u,v,w,v'\}$ where $w$ is the new vertex in $G_{e}$ and $v' \in \mathcal{L}_{G}\left( u\right)$. As $v, v'$ are leaves in $G_e$ adjacent to different support vertices,  $|D\cap S|=2$. But $(D-S)\cup \{ u\}$ is a dominating set of $G$ with $|(D-S)\cup \{ u\}|<|D|$, a contradiction.

\end{proof}

\begin{remark}\label{support-vertices}
Let $G$ be a graph and $e=uv$ be an edge of $G$ where $\{u,v\}\subseteq Supp(G)$. Then $\gamma (G-e)=\gamma (G)=\gamma (G_{e})$.
\end{remark}

\begin{definition} A graph $G$ is called a \emph{hairy graph} if every vertex of $G$ is a leaf or a support vertex.
\end{definition}
Examples of hairy graphs are stars, caterpillars and the corona $G\circ K_1$ of any graph $G$.
\begin{remark}\label{SopDomHairy}
Let $G$ be a hairy graph different from $K_2$. Then $Supp(G)$ is a minimum dominating set of $G$.
\end{remark}

\begin{theorem}\label{hairy}
If $G$ is a hairy graph with at least three vertices, then $G$ is an  SR-graph.
\end{theorem}

\begin{proof}
If $\gamma (G)=1$, then $G$ is a star and it is an SR-graph. Suppose $\gamma (G)\geq 2$.

%\vspace{.2cm}

By Remark $5$, $D=Supp(G)$ is a $\gamma $-set of $G$. Let $e=uv \in E(G)$. If $u,v\in D$, by Remark \ref{support-vertices}, $\gamma (G-e)=\gamma (G)=\gamma (G_{e})$.

Otherwise, we may suppose $u\in D$ and $v\in \mathcal{L}_{G}\left( u\right) $. As $G$ is
a connected graph and $\gamma (G)\geq 2$, $u$ is dominated by some vertex in $D$.

If $u$ is a weak support vertex, $D^{\prime }=(D-\{u\})\cup \{v\}$ is a dominating set
of $G-e$ and $G_{e}$ with $\left\vert D^{\prime }\right\vert =\left\vert
D\right\vert $, so by Remark \ref{obs1} $G$ is an SR-graph. In the other case, $u$ is a strong support vertex and by Lemma \ref{SR fuerte}, $\gamma
(G-e)=\gamma (G)+1=\gamma (G_{e})$ and $G$ is an SR-graph.
\end{proof}

\begin{corollary}\label{induced-hairy}
Every graph is an induced subgraph of an SR-graph.
\end{corollary}

\begin{proof}
Let $G$ be a graph. For $G=K_1$ or $G=K_2$ the result is clear. In the other case, by Theorem \ref{hairy}, the corona of $G$, $H=G\circ K_1$ is an SR-graph where $H[Supp(H)]=G$.
\end{proof}

\begin{definition} Let $H_1$ and $H_2$  be hairy graphs and let $u\in Supp(H_1)$ and $v\in Supp(H_2)$. For $t\geq 1$ we define a new graph $G_t(H_1,H_2)$ such that

\begin{itemize}
	\item $V(G_t(H_1,H_2))=V(H_1)\cup V(H_2)\cup \{ x_1,x_2,...,x_t\}$; and
	\item $E(G_t(H_1,H_2))= E(H_1)\cup E(H_2)\cup B$, where $B=\{ ux_1,x_1x_2,...,x_{t-1}x_t, x_tv\}$.
\end{itemize}

\end{definition}

The next theorem shows us a way to construct infinite many SR-graphs from two arbitrary hairy graphs.

\begin{theorem} Let $H_1$ and $H_2$  be hairy graphs with $\gamma (H_1)\geq 2$ and $\gamma (H_2)\geq 2$.
If $t=1$ or $t\equiv 0 \pmod 3$, then the graph  $G_t(H_1,H_2)$ is an SR-graph.
\end{theorem}

\begin{proof} By Lemma \ref{SR fuerte} and  Remark \ref{support-vertices}, we only need to analyze edges in the set $B$ or edges of the form $e=yz$ where $y$ is a weak support vertex and $z$ is a leaf. Let $D_1=Supp(H_1)$ and $D_2=Supp(H_2)$. By Remark \ref{SopDomHairy}, $D_1$ and $D_2$ are $\gamma$-sets of $H_1$ and $H_2$, respectively. Moreover,  $\gamma (G_{t}(H_1,H_2))=\gamma (H_{1})+\gamma (H_{2})+\gamma (P)$ where $P$ is the path $P=(x_2, ..., x_{t-1})$.

\

Let $D$ be a $\gamma $-set of $G_t(H_1,H_2)$ such that $D_1\cup D_2\subseteq D$. If $e=yz$ where $y$ is a weak support vertex and $z$ is a leaf, then $(D-\{y\})\cup \{z\}$ is a $\gamma$-set of $(G_t(H_1,H_2)-e)$ and $(G_t(H_1,H_2)_e)$.

\

In the rest of the proof we consider edges $e\in B=\{ ux_1,x_1x_2,...,x_{t-1}x_t, x_tv\}$.

\vspace{.3cm}

\textbf{Case 1. }If $t=1$, the edge $e\in \{ ux_1,x_1v\}$. Then $D=D_1\cup D_2$ is a $\gamma$-set of
$(G_t(H_1,H_2)-e)$ and $(G_t(H_1,H_2)_e)$.

\

\textbf{Case 2.} Let $t=3s$, $s\geq 1$. We have the following cases:

\begin{itemize}

\item If $e\in \{ ux_1,x_tv\},$ then for a $\gamma$-set $X$ of the path $(x_1,x_2,...x_t)$, the set $D=D_1\cup D_2\cup X$ is a $\gamma$-set of $G_t(H_1,H_2)$, $(G_t(H_1,H_2)-e)$ and $(G_t(H_1,H_2)_e)$.
\item If $e=x_1x_2$, then for a $\gamma$-set $X$ of $P$ such that $x_2\in X$, the set $D=D_1\cup D_2\cup X$ is a $\gamma$-set of $G_t(H_1,H_2)$, $(G_t(H_1,H_2)-e)$ and $(G_t(H_1,H_2)_e)$. Similarly, for the case of $e=x_{t-1}x_t$ consider a $\gamma$-set of $P$ such that $x_{t-1}\in X.$
\item If $e\in E(P)$, by Proposition $2$, $P$ is an SR-graph with $\gamma (P)= \gamma (P-e)=\gamma (P_e)$, which implies that $\gamma (G_t(H_1,H_2)-e)=\gamma ((G_t(H_1,H_2)_e)$.
\end{itemize}

\end{proof}

\section{Sub-removable trees}

In this section we give a characterization of trees which are SR-graphs. Those trees are called \emph{SR-trees}. Also, we give a characterization of bondage edges in SR-trees.

\begin{remark}
By Theorem \ref{hairy}, if a tree $T$ with at least three vertices  has diameter less or equal to three, then $T$ is an SR-graph.
\end{remark}

\begin{definition} Let $T$ be a tree and $e\in E\left( T\right) $.
\begin{enumerate}
\item The edge $e$ is a \textit{weak edge} of $T$ if $e\cap D=\emptyset $ for any $D$ in $\Gamma \left( T\right) $.
\item The edge  $e$ is a \textit{strong edge} of $T$ if $e$ satisfies Teschner's
Condition and there exists $D\in \Gamma \left( T\right) $ such
that $EPN\left( e\cap D , D \right) =\left\{ \overline{e\cap D}\right\}$.
\end{enumerate}
\end{definition}

\begin{remark}\label{Two in EPN}
If $e$ is a bondage edge and is not a strong edge of a tree $T$, then
$$
\left( N_T\left( e\cap D\right) -\left\{ \overline{e\cap D}\right\} \right)
\cap EPN\left( e\cap D,D\right) \neq \emptyset \text{ for any }D\text{ in }%
\Gamma \left( T\right).
$$
\end{remark}

\begin{remark}\label{dominantes}
Let $D$ be a dominating set of a graph $G$. If $e\in E(G)$ such that  $e\cap
D=\emptyset $, then $D\succ G-e$.
\end{remark}

In the next discussion, given a tree $T$ and an edge $e=uv$ of $T$, $T_{u}$
and $T_{v}$ denote the subtrees of $T-e$ which contain $u$ and $v$,
respectively.

\begin{theorem}\label{SR-treeCaracterizacion}
A tree $T$ is an SR-tree if and only if $T$ does not
contain neither weak nor strong edges.
\end{theorem}

\begin{proof} First we prove that if there is a weak or a strong edge in a tree $T$,
then $T $ is not an SR-tree.

\vspace{.3cm}
Let $D\in \Gamma \left( T\right) $. Suppose that $e=uv$ is a weak edge, then
$e\cap D=\emptyset $. By Remark \ref{dominantes},  $D\succ \left(
T-e\right) $, so $\gamma (T-e)=\gamma (T)$.  Suppose there exists a dominating set $D^{\prime }$ of $%
T_{e}$ such that $\left\vert D^{\prime }\right\vert =\left\vert D\right\vert
$. Let $w$ be the new vertex in $T_e$. If $w\notin D^{\prime }$, then $%
D^{\prime }$ belongs to $\Gamma \left( T\right) $ and $e\cap D^{\prime }\neq
\emptyset $, contradicting that $e$ is a weak edge. Otherwise, $w\in D^{\prime }$%
\ and $\left( D^{\prime }-\left\{ w\right\} \right) \cup \left\{ u\right\} $
is a $\gamma $-set of $T$ containing $u$, which contradicts that $e$ is
a weak edge. Therefore if $e$ is a weak edge, then $\gamma (T_e) > \gamma (T-e)$ and $T$ is not and SR -graph.

Suppose that $e$ is a strong edge of $T$. Then there exists
$D^{\prime }\in \Gamma \left( T\right) $ such that
$EPN\left( e\cap D^{\prime },D ^{\prime }\right) =\left\{ \overline{e\cap D^{\prime}}\right\}$. Therefore
$D=\left( D^{\prime }-\left\{ e\cap D^{\prime }\right\} \right)
\cup \left\{ w\right\}$ is a dominating set of $T_{e}$ and $\gamma \left( T_{e}\right) =\gamma \left( T\right) $. On the other hand,
by Theorem \ref{teschner}, $e$ is a bondage edge of $T$. So
$\gamma \left( T_{e}\right) <\gamma \left( T-e\right) $ and we conclude that
$T$ is not an SR-graph.

\vspace{.3cm}

Now we show that if there is neither weak nor strong edge in $T$, then $T$ is an SR-graph.

Let $e=uv$ be an edge of $T$. If there exists $D\in \Gamma
\left( T\right) $ such that $e\cap D=e$, then $\gamma(T-e)=\gamma \left( T\right)=\gamma(T_{e})$. So, we may
suppose that $\left\vert e\cap D\right\vert <2$ for any $D$ in $\Gamma
\left( T\right) $.

We consider two cases.

\textbf{Case 1}. There exists $D_{1}\in \Gamma \left( T\right) $ such that
$e\cap D_{1}=\emptyset $.

By Remark \ref{dominantes},  $D_{1}$ is a minimum dominating set of $T-e$. Since $e$ is not a weak edge, there
exists $D_{2}\in \Gamma \left( T\right) $ such that $\left\vert e\cap
D_{2}\right\vert =1$. Let $u=e\cap D_{2}$. We have partitions of $D_1=V_1\cup U_1$ and $D_2=V_2\cup U_2$ where
$V_1=V(T_v)\cap D_1, V_2=V(T_v)\cap D_2, U_1=V(T_u)\cap D_1$ and  $U_2=V(T_u)\cap D_2$. Since $T$ is a tree, we have the following relations:
$U_1\succ T_u, U_2\succ T_u$ and $V_1\succ T_v$.

If $|V_1|\leq |V_2|$, define $D=V_1\cup U_2$, then $|D|\leq |D_2|$.
Like  $V_1\succ T_v$ and  $U_2\succ T_u$, $D$ is a dominating set of $T-e$.
Therefore  is also a dominating set of $T$, so $\left\vert D\right\vert
=\left\vert D_{2}\right\vert $. Moreover,  $D\succ T-e$ and $u\in D$, which implies that $D\succ T_{e}$.

If $|V_1| > |V_2|$, define the set of vertices $D=U_1\cup \{ v\} \cup V_2$, then $|D|\leq |D_1|$.
Like  $U_1\succ T_u$ and  $(V_2\cup \{ v\})\succ T_v$, $D$ is a dominating set of $T_e$.

Therefore, in this case $\gamma(T-e)=\gamma \left( T\right)=\gamma(T_{e})$.

\textbf{Case 2}. For any $D$ in $\Gamma \left( T\right) $, $\left\vert e\cap D\right\vert =1$.

If there exists $D$ in $\Gamma \left( T\right) $ such that $\overline{e\cap D}\notin EPN\left( e\cap D,D\right) $,
then $D\succ T-e$ and $D\succ T_{e}$. Therefore we may suppose that $e$ satisfies Teschner's Condition and by
Theorem \ref{teschner} we have $\gamma \left( T-e\right) >\gamma \left( T\right) $.

Let $\gamma (T)=s$. Suppose there exists $D^{\prime }$ a $\gamma $-set of $T_{e}$ such that
$\left\vert D^{\prime }\right\vert =s$. Recall $e=uv$ and let $w$ be the new vertex in $T_{e}$.

If $w\notin D^{\prime }$, then $D^{\prime }\succ T-e$, which contradicts
$\gamma \left( T-e\right) >\gamma \left( T\right) $. In the other case,
$w\in D^{\prime }$, $D=\left( D^{\prime }-\left\{ w\right\} \right) \cup
\left\{ u\right\} $ is a dominating set of $T$ and
$u\notin D^{\prime }$ because $\left\vert D\right\vert \geq s$.

So, $\left\vert D\right\vert =s$, $D\in \Gamma \left( T\right) $, $e\cap
D=u$ and $v\in EPN\left( e\cap D,D\right)$. Moreover, since
$u\notin D^{\prime }$ and $D^{\prime }\succ T_{e}$, we have that for any $x$
in $N_T\left( u\right) -\left\{ v\right\} $ there exists $d\in \left\{
D-\left\{ u\right\} \right\} $ such that $d\succ x$. Therefore $\left( N_T\left( u\right) -\left\{ v\right\}
\right) \cap EPN\left( u,D\right) =\emptyset $ which contradicts Remark \ref{Two in EPN}.

Therefore $\gamma \left( T-e\right) =\gamma \left( T\right) +1=\gamma \left(
T_{e}\right) $.
\end{proof}

\vspace{.3cm}

The next theorem gives a characterization of bondage edges in SR-trees.

\bigskip

\begin{theorem} For an SR-tree $T$ and $e\in E\left( T\right)$,
$e$ is a bondage edge of $T$ if and only if one of the ends of $e$ is a leaf and the other is a strong support.
\end{theorem}

\begin{proof} Let $e=uv$. If $u$ is a strong support vertex and $v$ is a leaf, then by Lemma \ref{SR fuerte},  $e$ is a bondage edge of $T$.

Conversely, suppose $e$ is a bondage edge of $T$. By Theorem \ref{teschner}, $\left\vert e\cap D\right\vert =1$
and $\overline{e\cap D}\in EPN\left( e\cap D,D\right) $ for any $D$ in $\Gamma
\left( T\right) $. We consider two cases.

\textbf{Case 1. }There exist $D_{1},D_{2}\in $ $\Gamma \left( T\right) $
such that $e\cap D_{1}=u$ and $e\cap D_{2}=v$.

Let $V_1= V\left( T_{v}\right) \cap D_{1}$, $V_2= V\left( T_{v}\right) \cap D_{2}$, $U_1=V\left( T_{u}\right)\cap D_{1}$ and
$U_2=V\left( T_{u}\right)\cap D_{2}$.

Suppose $|V_1|<|V_2|$. Since $T$ is a tree, $V_1\succ T_{v}-\{v\}\ $ and $U_2\succ T_{u}-\{ u\}$. Therefore
$D=U_2\cup V_1\cup \{ v\}$ is a dominating set of $T$.

As $|D|=|U_2|+|V_1|+1<|U_2|+|V_2|+1=|D_2|+1$, we have $|D|\leq|D_2|$ and hence $D\in \Gamma (T).$

By Theorem \ref{teschner}, $u\in EPN\left( v,D\right)$. Moreover, $V_1\succ T_{v}-\{v\}\ $,
therefore $EPN\left( v,D\right)=\{u\}$. If $w$ denotes the new vertex in $T_e$, then
$D^{\prime }=\left(D-\left\{ v\right\} \right) \cup \left\{ w\right\} $ is a dominating set of
$T_{e}$ such that  $\left\vert
D^{\prime }\right\vert =\left\vert D\right\vert =\gamma \left( T\right)<\gamma (T-e) $, contradicting that $T$ is an SR-tree.

If $|V_1|\geq|V_2|$, consider $D=U_1\cup V_2,$ a dominating set of $T$. Then
$|D_1|=|U_1|+|V_1|\geq |U_1|+|V_2|=|D|$, which implies that $D$ is a $\gamma $-set of $T$ with $e\cap D=e$, a contradiction with the definition of a bondage edge.

\textbf{Case 2.} For any $D$ in $\Gamma \left( T\right)$, $e\cap D=u$.

Suppose $\left\vert N_T\left( v\right) \right\vert \geq 2$ and let $x\in \left\{
N_T\left( v\right) -\left\{ u\right\} \right\} $. By Theorem \ref{teschner}, $v\in EPN\left( u,D\right) $ for any $D$ in
$\Gamma \left( T\right) $, therefore $x\notin D$ for any $D$ in $\Gamma \left( T\right) $. Hence for the edge
$\widetilde{e}=vx$ we have $\widetilde{e}\cap D=\emptyset $ for any $D$ in $%
\Gamma \left( T\right) $ i.e., $\widetilde{e}$ is a weak edge of $T$. Therefore,
by Theorem \ref{SR-treeCaracterizacion}, $T$ is not and SR -tree, a contradiction. So, $v$ is a leaf of $%
T$.

Note that $d_{T}\left( u\right) \geq 2$. If some vertex of $N\left( u\right) -\left\{ v\right\}$ is a leaf,  then $u$ is a
strong support and we are done. Otherwise, let $N\left( u\right) -\left\{ v\right\} =\left\{ x_{1},...,x_{r}\right\} ,$
$r\geq 1$ and $T_{x_{i}}$ be the subtree of $T-x_{i}u$ containing $x_{i}$, $1\leq
i\leq r$. Suppose that for any $1\leq i\leq r$ there exists $D_{i}$ in $\Gamma \left( T\right) $ such that $%
x_{i}\notin EPN\left( u,D_{i}\right) $. Since $T$ is a tree, $\left(
D\cap V\left( T_{x_{1}}\right) ,D\cap V\left( T_{x_{2}}\right) ,...,D\cap
V\left( T_{x_{r}}\right) ,u\right) $ is a partition of $D$ for any\ $D$ in $%
\Gamma \left( T\right) $.

Let $D^{\prime },D^{\prime \prime }\in \Gamma \left( T\right) $. If $%
\left\vert D^{\prime }\cap V\left( T_{x_{j}}\right) \right\vert >\left\vert
D^{\prime \prime }\cap V\left( T_{x_{j}}\right) \right\vert $ for some $j$,
then $D=\bigcup\limits_{i\neq j}\left( D^{\prime }\cap V\left(
T_{x_{i}}\right) \right) \cup \left( D^{\prime \prime }\cap V\left(
T_{x_{j}}\right) \right) \cup \left\{ u\right\} $ is a dominating set of $T$
where $\left\vert D\right\vert <\left\vert D^{\prime }\right\vert $, which
is impossible. Therefore $\gamma \left( T\right)
=\sum\limits_{i=1}^{r}\left\vert D_{i}\cap V\left( T_{x_{i}}\right)
\right\vert +1$ and $D=\bigcup\limits_{i=1}^{r}\left( D_{i}\cap V(T_{x_{i}})\right)
\cup \left\{ u\right\} $ is a $\gamma $-set of $T$ which satisfies $\left(
N\left( u\right) -\left\{ v\right\} \right) \cap EPN\left( u,D\right)
=\emptyset $, what contradicts Remark \ref{Two in EPN}. Therefore there exists $v^{\prime
}\in \left( N\left( u\right) -\left\{ v\right\} \right) $ such that $%
v^{\prime }\in EPN\left( u,D\right) $ for any $D$ in $\Gamma \left( T\right) $%
.
Finally, in the same way that we proved $v$ is a leaf, we can prove that $v^{\prime }$ is also a leaf. Therefore $u$ is a strong support of $T$.

\end{proof}

In some cases, could be useful to rewrite the above theroem as

\begin{theorem}
For an SR-tree $T$ and $e\in E\left( T\right)$, $%
\gamma \left( T-e\right) = \gamma \left( T\right)
$ if and only if no ends of $e$ is a leaf or one of the ends of $e$ is a weak  support.
\end{theorem}

\section{Anti-sub-removable graphs}

In this section we characterize ASR-graphs with domination number one, give some properties of ASR-graphs, show that ASR-graphs are $\gamma$-insensitive and give an infinity family of ASR-graphs with an arbitrary domination number.

Since $P_2$ is an ASR-graph, from now, we assume that $|V(G)|\geq 3$ for any graph $G.$

\begin{remark}\label{dominacion1} Let $G$ be a graph. If $\gamma (G)=1,$ then $\gamma \ (G_e)=2$ for any edge $e\in E(G)$.
\end{remark}

\begin{lemma}\label{no-SR-ASR}
If $G\neq K_{1,n}$ is a graph with exactly one or two universal vertices, then $G$ is neither SR nor ASR-graph.
\end{lemma}

\begin{proof}
By Remark \ref{dominacion1},  $\gamma\ (G_e)=2$ for any edge $e\in E(G)$.

Suppose $G$ has a unique universal vertex $x$. As $G$ is not a star, there exist vertices $y,z$ in $V(G)$ such that $e=xy, f=yz$ are edges of $G$.
Then $\gamma \ (G-e)=2$ and $\gamma \ (G-f)=1$. So, in this case, $G$ is neither  SR  nor ASR-graph.

Otherwise $G$ has exactly two universal vertices $x,y$. Let $e\neq xy\in E(G)$, then $\gamma \ (G-xy) =2$ and $\gamma \ (G-e)=1$. Therefore, $G$ is neither SR nor ASR-graph.
\end{proof}

\begin{lemma}\label{ASR-3u} If $G$ is a graph with at least three universal vertices, then $G$ is an ASR-graph.
\end{lemma}

\begin{proof}
Let $e$ be an edge of $G$. By hypothesis, the graph $G-e$  has at least one universal vertex, so  $\gamma \ (G-e)=1$ and by Remark \ref{dominacion1},  $\gamma\ (G_e)=2$. Therefore $G$ is an ASR-graph.
\end{proof}

Given two vertex-disjoint graphs $G$ and $H$, the sum $G+H$ is
the graph with vertex set $V\left( G\right) \cup V\left( H\right) $ and edge
set $E\left( G\right) \cup E\left( H\right) \cup \left\{
xy :x\in V\left( G\right) \text{ and }y\in V\left( H\right)
\right\} $.

As a direct consequence of Lemmas \ref{no-SR-ASR} and \ref{ASR-3u} we have the following characterization of ASR-graphs with order at least three and domination number one.

\begin{theorem}\label{ASRg=1}
A graph $G$ with $\gamma \left( G\right) =1$ is an ASR-graph if and
only if there exists a (possible null) graph $H$ such that $G=K_{3} + H$.
\end{theorem}

\begin{corollary}
Every graph is an induced subgraph of an ASR-graph.
\end{corollary}

\begin{lemma}\label{particion} If $G$ is an ASR-graph, then for any $\gamma $-set $D=\{ x_{1},x_{2},...,x_{p}\}$ of $G$ we have $( N[x_1], N[x_2], ..., N[x_p] )$ is a partition of $V(G)$.
\end{lemma}

\begin{proof} For $1\leq i\leq p$, $N[x_i]\neq \emptyset $  and $V(G)= \bigcup\limits_{i=1}^{p} N[x_i]$.  Suppose there exists a vertex $y\in N[x_i]\cap N[x_j]$ for some $i\neq j$. Hence, for the edge $e=x_iy$, it is clear that $D\succ G-e$ and $D\succ G_{e}$, which contradicts that $G$ is an ASR-graph.
\end{proof}

Observe that this lemma implies that if $G$ is an ASR-graph, then every $\gamma $-set of $G$ is an independent set. The converse of this lemma is not true (see Figure 1).
\begin{figure}[h]
   \centering
    \includegraphics[width=1\textwidth]{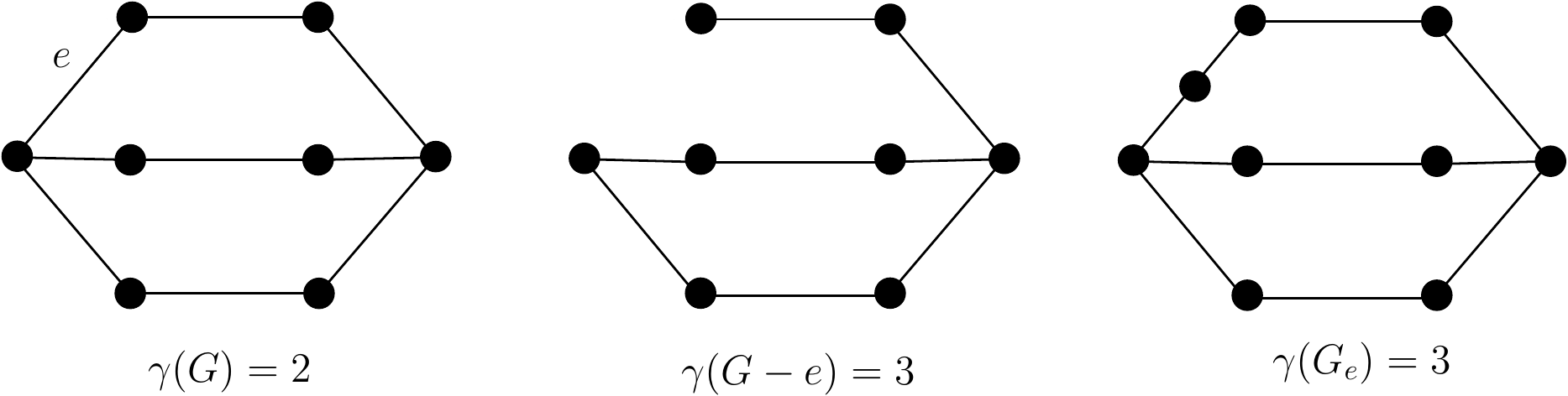}
        \caption{The converse of Lemma \ref{particion} is not true. }
  \label{Figure 1}
\end{figure}
%\vspace{1cm}

\begin{remark}\label{NoLeaf}
Every graph $G$ has a $\gamma$-set which not contains a leaf of $G$.
\end{remark}

\begin{theorem}\label{ASR no bondages} An ASR-graph has no bondage edges.
\end{theorem}

\begin{proof} Let $G$ be a connected ASR-graph. If $\gamma (G)=1$, by Remark \ref{dominacion1},  $G$ has no bondage edges. So we may assume that $G$ has order at least $4$ and $\gamma
\left( G\right)=p \geq 2$.

Suppose that $e=uv$ is a bondage edge of $G$, i.e.,   $\gamma \left( G-e\right) >\gamma \left( G\right) $. By Theorem \ref{teschner},  $|e\cap D|=1$ for any
$D \in  \Gamma (G)$. Let $D=\{ u, x_{2},..., x_{p}\}$   be a $\gamma $-set of $G$, as $G\neq K_2$ and by Remark \ref{NoLeaf} we may assume $d_G(u) \geq 2$. Since $G$ is an ASR-graph, there exists $D^{\prime }\succ G_{e}$ such that
$|D^{\prime}|=p. $

By Lemma \ref{particion}, $N_{G_{e}}\left[ x_{i}\right] =N_{G}\left[ x_{i}\right] $ for any $i\geq 2$ and
$\left\vert \bigcup\limits_{i=2}^{p}\left( D^{\prime }\cap N_{G_{e}}\left[ x_{i}\right] \right) \right\vert \geq p-1$.

Observe that if $\left\vert D^{\prime }\cap N_{G_{e}}\left[ x_{i}\right] \right\vert \geq 2$ for some
$i\geq 2$, then $D^{\prime }\cap N_{G_{e}}\left[ u\right] =\emptyset $, which contradicts that $D'$ is a dominating set of $G_e$.
Hence $\left\vert D^{\prime }\cap N_{G_{e}}\left[ x_{i}\right] \right\vert =1$ for any $i\geq 2$.
Let $z_{i}=D^{\prime }\cap N_{G_{e}}\left[ x_{i}\right] $, $2\leq i\leq p$ and $Z=\left\{ z_{2},...,z_{p}\right\} $.

Let $w$ be the new vertex in $G_{e}$. Since $\left\vert Z \right\vert=p-1$ and $Z \not\succ w$ we have $\left\vert D^{\prime }\cap \left\{ u,v,w\right\}
\right\vert =1$. On the other hand, in $G_{e}$, $Z\nsucc u$ by Lemma \ref{particion} and $%
v\nsucc u$, so $D^{\prime }=Z\cup \left\{ w\right\} $.
Since $Z\succ \left\{ N_{G_{e}}\left( u\right) -\left\{ w\right\} \right\} $ and $d_{G_{e}}\left( u\right) \geq 2$, there exists a vertex
$y\in N_{G_{e}}\left( u\right) \cap N_{G_{e}}\left( z_{i}\right) $ for some $i\geq 2$. Therefore for $f=yz_{i}$,
the set $Z\cup \left\{ u\right\} $ dominates $G-f$ and $G_f$,
which is a contradiction.

\end{proof}

The next corollary is an immediate consequence of Theorem \ref{ASR no bondages}.

\begin{corollary}
Every ASR-graph is $\gamma$-insensitive.
\end{corollary}

\begin{lemma}
An ASR-graph has no leaves.
\end{lemma}

\begin{proof} Suppose that $e=uv$ is an edge of an ASR-graph $G$ such that
$d_G\left( u\right) =1$. By Theorem \ref{ASR no bondages} $e$ is not a
bondage edge.  If $D\cap e=\{ v\}$ for any $D\in \Gamma (G)$, then $e$ satisfies Teschner's Condition and by Theorem \ref{teschner} the edge $e$ is a bondage edge, a contradiction. Therefore,  there exists
$D'\in \Gamma \left( G\right) $ such that $e\cap D'=\{ u\}$.
Hence $D=\left( D'-\left\{ u\right\} \right) \cup \left\{ w\right\}$, where $w$ is the new vertex in $G_e$, is a dominating set of $G_{e}$ such that $|D'|=|D|=\gamma (G-e)$
and this contradicts that  $G$ is an ASR-graph.
\end{proof}

\begin{corollary} There is no ASR-tree except $P_2$.
\end{corollary}

Given a vertex-disjoint graphs $H_1,H_2, ..., H_m$, we denote by $E(H_1,H_2,...,H_m)$ the set of all possible edges between them, that is, the set of edges of the complete $m$-partite graph determined by $\left( V(H_1),V(H_2),...,V(H_m) \right)$.

\begin{definition}
Let $m\in \mathbb{N}$. We say that a graph $G$ belongs to the family of graphs $\mathcal{B}_{m}$ if there exist $m$ vertex-disjoint ASR-graphs
$G_{1},G_{2},...,G_{m}$ of order at least three and domination number one, such that
\begin{enumerate}
\item $V(G)=\bigcup \limits_{i=1}^{m} V(G_i)$.
\item $E(G) =\bigcup \limits_{i=1}^{m}E(G_i)\cup \widetilde{E}(G)$.
\end{enumerate}
Where
\begin{itemize}
\item For $1\leq i\leq m$, $G_i=K_{r_i}+H_i$, where $r_i$ is the number of universal vertices in $G_i$.
\item For $1\leq i\leq m$, $S_{i}$ is a subset of $V(H_{i})$ such that $N_{H_{i}}\left[ S_{i}\right] \neq V\left( H_{i}\right)$.
\item $\widetilde{E}(G)\subseteq E(H_1[S_1],H_2[S_2],...,H_m[S_m])$.
\end{itemize}
\end{definition}

\begin{proposition}
Let $m\in \mathbb{N}$. Any graph in $\mathcal{B}%
_{m} $ is an ASR-graph with domination number $m$.
\end{proposition}

\begin{proof}
Let $G\in \mathcal{B}_{m}$. By Theorem \ref{ASRg=1}, $r_{i}\geq 3$ for $1\leq i\leq m$ and by definition of $G$, we have
the following remarks.

\begin{remark}\label{DominGi}
Let $e\in E(G)$. If $D$ is a dominating set of $G_{e}$, then $D\cap V(G_i) \neq \emptyset$ for $1\leq i\leq m$.
\end{remark}

\begin{remark}\label{Dom2inGi}
Let $e\in E(G)$, $D$ be a dominating set of $G_{e}$ and $w$ be the new vertex in $G_e$. If $D\cap S_j\neq \emptyset $ for some $j$ and $w\notin D$, then $\left\vert D\cap V(G_j)\right\vert > 1$.
\end{remark}

By Theorem \ref{ASRg=1}, we only need to prove the result for $%
m\geq 2$.

It is clear that $D=\{ x_1,x_2, ...,x_m\}$, where $x_i\in V(K_{r_i})$ is a $\gamma $-set of $G$. Note that for any $e\in E(G)$ the set $D=\{ x_1,x_2,...,x_m\}$, where $x_i$ is an universal vertex of $G_i$ and $D\cap e=\emptyset$, is a dominating set of $G-e$. Thus $\gamma (G-e)=\gamma (G)=m$.

Let $e=uv$ be an edge of $G$, $w$ be the new vertex in $G_{e}$ and $D$ be a $\gamma $-set of $G_{e}$. If $w\in D$, then by Remarks \ref{DominGi} and \ref{obs1}, $\left\vert D\right\vert =m+1$. Otherwise we may assume that $u\in D\cap V(G_{j})$ for some $j$. Moreover, if $\left\vert D\cap V(G_i)\right\vert > 1$ for some i, by Remarks \ref{DominGi} and \ref{obs1} we have $\gamma\left( G_{e}\right) =m+1$ and we are done. So, we may suppose $\left\vert D\cap V(G_i)\right\vert = 1$ for $1\leq i\leq m$.

If $e\in \widetilde{E}(G)$, then $u \in S_j$ and by Remark \ref{Dom2inGi} we have $\left\vert D\cap V(G_j)\right\vert > 1$, a contradiction. Otherwise $e\in E(G_{j})$. Since $G_j$ is an ASR-graph, $u$ is not a dominating set of $\left( G_{j}\right) _{e}$, but $\left\vert D\cap V(G_j)\right\vert = 1$. Therefore there exist $x\in S_j$ and $y\in S_k$ for some $k \neq j$ such that $y \in D$ and $y \succ x$. Again, by Remark \ref{Dom2inGi} $\left\vert D\cap V(G_k)\right\vert > 1$, which is impossible.
\end{proof}

\textbf{Acknowledgements}

The authors thank the financial support received from Grant UNAM-PAPIIT IN-117812 and SEP-CONACyT.

\end{document}